\documentclass[12pt,reqno]{amsart}
\usepackage{amsmath,amsthm,amssymb,amsfonts,amscd}
\usepackage{mathrsfs}%mathtools
\usepackage{bbm}
\usepackage{bbding}
\usepackage[backref=page]{hyperref}
\usepackage{geometry}
\usepackage{color}
\usepackage{xcolor}

\usepackage{picture,epic}
\usepackage{tikz}

\usepackage{etoolbox}
\usepackage{orcidlink}

\newif\ifrevisionmode
\revisionmodetrue
\numberwithin{equation}{section}

\theoremstyle{plain}
\newtheorem{theorem}{Theorem}[section]
\newtheorem{lemma}[theorem]{Lemma}

\theoremstyle{definition}
\theoremstyle{remark}
\newtheorem{remark}[theorem]{Remark}

\renewcommand{\Re}{\operatorname{Re}}
\renewcommand{\Im}{\operatorname{Im}}

\newcommand{\sym}{\operatorname{sym}}

\newcommand{\GL}{\operatorname{GL}}

\newcommand{\SO}{\operatorname{SO}}

\newcommand{\Sp}{\operatorname{Sp}}
\newcommand{\GSp}{\operatorname{GSp}}

\newcommand{\dd}{\mathrm{d}}

   \DeclareFontFamily{U}{wncy}{}
    \DeclareFontShape{U}{wncy}{m}{n}{<->wncyr10}{}
    \DeclareSymbolFont{mcy}{U}{wncy}{m}{n}
    \DeclareMathSymbol{\Sh}{\mathord}{mcy}{"58}

\makeatletter
\def\@tocline#1#2#3#4#5#6#7{\relax
  \ifnum #1>\c@tocdepth % then omit
  \else
    \par \addpenalty\@secpenalty\addvspace{#2}%
    \begingroup \hyphenpenalty\@M
    \@ifempty{#4}{%
      \@tempdima\csname r@tocindent\number#1\endcsname\relax
    }{%
      \@tempdima#4\relax
    }%
    \parindent\z@ \leftskip#3\relax \advance\leftskip\@tempdima\relax
    \rightskip\@pnumwidth plus4em \parfillskip-\@pnumwidth
    #5\leavevmode\hskip-\@tempdima
      \ifcase #1
       \or\or \hskip 1em \or \hskip 2em \else \hskip 3em \fi%
      #6\nobreak\relax
    \hfill\hbox to\@pnumwidth{\@tocpagenum{#7}}\par% <---- \dotfill -> \hfill
    \nobreak
    \endgroup
  \fi}
\makeatother

\begin{document}

\title[Symmetric cubes]
{Decorrelation of $\mathrm{SO}(5)\times \mathrm{SO}(2)$ Bessel periods for symmetric cubes of $\mathrm{GL}(2)$}
\author{Shenghao Hua~\orcidlink{0000-0002-7210-2650}}
\address[1]{Data Science Institute
and School of Mathematics \\ Shandong University \\ Jinan \\ Shandong 250100 \\China}
\email{huashenghao@vip.qq.com}
\curraddr
{\itshape EPFL-SB-MATH-TAN
Station 8,
1015 Lausanne, Switzerland}

\author{Xinchen Miao~\orcidlink{0009-0002-8480-274X}}
\address[2]{Mathematisches Institut, Endenicher Allee 60, Bonn, 53115, Germany}
\email{miao@math.uni-bonn.de, olivermiaoxinchen@gmail.com}

%\date{\today}

\begin{abstract}
We demonstrate that for symmetric cubes of algebraic regular Hecke eigenforms on $\mathrm{GL}(2)$, a decorrelation phenomenon occurs for global Bessel periods of $\mathrm{SO}(5)\times \mathrm{SO}(2)$ averaged over imaginary quadratic fields, conditional on the Generalized Riemann Hypothesis.
\end{abstract}



\maketitle
%\setcounter{tocdepth}{1}%ʹĿ¼ֻ��ʾ����
%\tableofcontents

%%%%%%%%%%%%%%%%%%%%%%%%%%%%%%%%%%%%%%%%%%%%%%%%%%%%%%%%%%%%%%%%
%%%%%                        Section                       %%%%%
%%%%%%%%%%%%%%%%%%%%%%%%%%%%%%%%%%%%%%%%%%%%%%%%%%%%%%%%%%%%%%%%
\section{Introduction} \label{sec:Intr}

The Gan--Gross--Prasad conjecture studies the branching behavior of automorphic representations of certain classical groups and predicts that when a representation of a larger group is restricted to a subgroup, the dimension of certain Hom spaces is precisely related to the value of the corresponding $L$-function at special points; the Bessel period is a special type of period defined by an integral that reflects the occurrence of representations in certain Bessel models and is often used to characterize the nonvanishing of $L$-functions; all these belong to the broader framework of the Langlands program which aims to unify number theory, algebraic geometry and harmonic analysis through the connections among symmetry representation and $L$-functions; see \S\ref{sec:bessel} for further details.

For an even weight $k$ algebraic regular (in the sense as in Clozel~\cite{Clozel2015}) Hecke eigenform $f$ of level $N$, it follows from the work of Kim and Shahidi~\cite{KimShahidi2001, KimShahidi2002a, KimShahidi2002b} that $\sym^3 f$ is a self-dual cuspidal automorphic representation of $\GL(4)$ with trivial central character, and may also be viewed as a representation of $\mathrm{PGSp}(4)\simeq \mathrm{SO}(3,2)$. 
Furthermore, Banerjee--Mandal--Mondal~\cite{BMM2005} and Conti~\cite{Conti2019} showed that the level $M$ of $\sym^3 f$ has all its prime divisors contained among those of $N$. 
Ramakrishnan and Shahidi~\cite{RamakrishnanShahidi2007} proved that for an elliptic curve $E$, the representation $\sym^3 E$ corresponds to a Siegel modular form, while Loeffler and Zerbes~\cite{LoefflerZerbes2023} established that $\sym^3 f$ is a vector-valued Siegel modular form of weight $(2k-1, k+1)$.

In this paper, we demonstrate that for the symmetric cubes of algebraic regular Hecke eigenforms, one obtains a \emph{decorrelation phenomenon} for global Bessel periods averaged over imaginary quadratic fields, conditional on the Generalized Riemann Hypothesis (GRH).
This phenomenon does not hold for all global Bessel periods of $\mathrm{SO}(5)\times \mathrm{SO}(2)$, such behavior is predicted by the Keating--Snaith conjecture, what is special about the symmetric cube is that its family of quadratic twisted $L$-functions is of orthogonal type. 
In particular, for a Rankin--Selberg form $\varphi$, this is (and must be) false, as predicted by the Keating--Snaith conjecture for quadratic-twisted families of $L$-functions associated with $\GL_2 \times \GL_2$ forms of symplectic type, since such families are naturally expected to exhibit very large values.
For its connection with moments of $L$-functions, see \S\ref{sec:quadratic}; for a general discussion of families of $L$-functions and their symmetry types, see Soundararajan's 2022 ICM proceedings article~\cite{Soundararajan2021}.

\begin{theorem}[Decorrelation of global Bessel periods]\label{thm:indep}
Let $1 \le i \le m$, and suppose each $f_i$ is an even-weight $k_i$ algebraic regular Hecke eigenform, with level $N_i$.
Assume the forms $f_1, \dots, f_m$ are pairwise distinct, and each symmetric cube $\sym^3 f_i$ has level $M_i$.
Let $M_0 = [8, M_1, \dots, M_m]$.
Let $a \bmod{M_0}$ denote a residue class with $a \equiv 1 \pmod{4}$ and $(a, M_0) = 1$.

Assume the GRH holds for
\begin{equation*}
L(s,\chi_d), \quad
L(s, \mathrm{sym}^k f_i), \quad
L(s, \mathrm{sym}^3 f_i\times \chi_d), \quad
L(s, \sym^3 f_i \times \sym^3 f_j)
\end{equation*}
for all $k=2,4,6$, all $f_i\neq f_j$ and for all $D \le - d \le 2D$ with $d \equiv a \pmod{M_0}$.

Let $\varphi_i\in \sym^3 f_i$.
Then for any non-trivial additive characters $\psi_1, \dots, \psi_m$ of $\mathbb{Q}\backslash \mathbb{A}_\mathbb{Q}$, and any $\varepsilon > 0$, we have
\begin{equation}\label{eqn:decorperoid}
\frac{1}{D}
\sum_{\substack{K=\mathbb{Q}(\sqrt{d})\\
D \le -d \le 2D \\ \mu(|d|)^2 = 1\\
d \equiv a \pmod{M_0}}}
\prod_{i=1}^{m}
|B_K(\varphi_i,1;\psi)|
\ll_{f_1,\dots,f_m,\psi_1, \dots, \psi_m,\varepsilon}
(\log D)^{-\frac{m}{8} + \varepsilon}
\end{equation}
as $D \to \infty$, where $B_K(\varphi_i,1;\psi)$ denotes the global Bessel periods as defined in \eqref{eqn:defofbessel}.
\end{theorem}

We prove Theorem~\ref{thm:indep} in \S\ref{sec:proofofmain}.

\section{Bessel periods for \texorpdfstring{$\mathrm{SO}(5)\times \mathrm{SO}(2)$}{SO(5), SO(2)}}\label{sec:bessel}

\subsection{Automorphic periods}

Let $\Pi$ be a cuspidal automorphic representation of a reductive group $\mathrm{G}(\mathbb{A}_F)$, where $\mathbb{A}_F$ denotes the adele ring of a fixed number field $F$. For $\phi \in \Pi$ and a closed subgroup $H \subseteq \mathrm{G}$, one of the central problems in the theory of automorphic forms, $L$-functions, and the Langlands program is to study the $H$-period
\[
   \int_{H(F) \backslash H(\mathbb{A}_F)} \phi(h) \, \dd h,
\]
for a suitable choice of $H$. 
In many significant cases of the pair $(G,H)$, this $H$-period is closely and beautifully related to the central value of the automorphic $L$-function $L\!\left(1/2, \Pi\right)$.

The first example can be traced back to E.~Hecke. Let $\pi$ be a cuspidal automorphic representation of $\mathrm{PGL}_2(\mathbb{A}_F)$. Let $T$ be the diagonal subgroup, which can be embedded naturally into $\mathrm{PGL}_2$. Then, for every $\varphi \in \pi$, we have an identity of the following shape:
\[
   \int_{T(F)\backslash T(\mathbb{A}_F)} 
      \varphi \!\begin{pmatrix} a & \\ & 1 \end{pmatrix} \, \dd a 
   \;\approx\; L\!\left(\tfrac{1}{2}, \pi \right),
\]
where $L(s,\pi)$ denotes the degree-two standard $L$-function of $\pi$, and the symbol ``$\approx$'' indicates that the equality holds up to certain nonzero multiplicative factors.

Another important example is the celebrated Waldspurger formula. Let $E/\mathbb{Q}$ be a quadratic extension. In the 1980s, Waldspurger established a remarkable formula for the central value of the base-change $L$-function
\[
   L(s,\pi_E) = L(s,\pi)\, L(s,\pi \times \chi_E),
\]
where $\chi_E$ is the idele class character associated with the extension $E/\mathbb{Q}$. Waldspurger’s formula takes, roughly speaking, the following form:
\[
   \left| \int_{T(\mathbb{Q}) \backslash T(\mathbb{A})} 
      \varphi(t)\, \dd t \right|^2 
   \;\approx\; L\!\left( \tfrac{1}{2}, \pi_E \right),
\]
for $\varphi \in \pi$, where $T$ is a torus in $\mathrm{PGL}_2$ isomorphic to 
$\mathrm{Res}_{E/F}(\mathbb{G}_m)/\mathbb{G}_m$.

This result has led, in subsequent years, to striking arithmetic applications, including deep connections with the Birch--Swinnerton-Dyer conjecture and the theory of $p$-adic $L$-functions.

A natural generalization of Hecke's integral is provided by the theory of Rankin--Selberg convolutions, as developed by Jacquet, Piatetski-Shapiro, and Shalika \cite{JPSS1983}. Let $\mathrm{G}=\GL(n+1)$ and $H=\GL(n)$, where $H$ embeds into $\mathrm{G}$ in the natural way. For two cuspidal automorphic representations $\pi_1$ and $\pi_2$ of $\GL_{n+1}(\mathbb{A}_F)$ and $\GL_n(\mathbb{A}_F)$ respectively, we have
\[
   \int_{\GL_n(F)\backslash \GL_n(\mathbb{A}_F)} 
      \varphi\!\begin{pmatrix} h & \\ & 1 \end{pmatrix} \, \sigma(h)\, \dd h 
   \;\approx\; L\!\left(\tfrac{1}{2}, \pi_1 \times \pi_2 \right),
\]
where $\varphi \in \pi_1$, $\sigma \in \pi_2$, and $L(s,\pi_1 \times \pi_2)$ is the Rankin--Selberg $L$-function for $\GL(n+1)\times \GL(n)$.

On the other hand, the Gan--Gross--Prasad conjectures provide a far-reaching higher-rank generalization of Waldspurger’s result. Note that $\SO(2,1) \cong \mathrm{PGL}_2$ and $\mathrm{Res}_{E/F}(\mathbb{G}_m)/\mathbb{G}_m \cong \SO(1,1)$. Thus, Waldspurger’s formula naturally suggests a generalization relating automorphic $L$-functions on $\SO(n+1)\times \SO(n)$ to certain automorphic periods. This is the original version of the Gan--Gross--Prasad conjectures. Later, other versions (see \cite{Harris2014} for the case $\mathrm{U}(n+1) \times \mathrm{U}(n)$) were formulated, which enriched the theory substantially (\cite{GGP2012a, GGP2012b, GGP2023}).

For instance, let $E/F$ be a quadratic extension of number fields, and let $W \subseteq V$ be Hermitian spaces over $E$ such that the orthogonal complement $Z = W^{\perp}$ in $V$ is one-dimensional. Let $N \subseteq \mathrm{U}(V)$ be the unipotent radical of a parabolic subgroup $P \subseteq \mathrm{U}(V)$. Consider the triple $(G,H,\xi)$, where 
\[
   G = \mathrm{U}(V) \times \mathrm{U}(W), 
   \quad H = \mathrm{U}(W)\ltimes N \subseteq G,
\]
and $\xi: N(\mathbb{A}_F)\to \mathbb{C}^{\times}$ is a character trivial on $N(F)$ and extending to a character of $H(\mathbb{A}_F)$ that is trivial on $\mathrm{U}(W)(\mathbb{A}_F)$. If $\pi = \pi_V \otimes \pi_W$ is a cuspidal automorphic representation of $G(\mathbb{A}_F)$, we define the automorphic period
\[
   \mathcal{P}_{H,\xi}(\varphi) 
   = \int_{H(F)\backslash H(\mathbb{A}_F)} \varphi(h)\, \xi(h)\, \dd h,
   \qquad \varphi \in \pi.
\]
Let $\pi_E = \pi_{V,E} \otimes \pi_{W,E}$ denote the weak base change of $\pi$ to $\GL_{n+1}(\mathbb{A}_E)\times \GL_n(\mathbb{A}_E)$, and write
\[
   L(s,\pi_E) := L(s,\pi_{V,E}\times \pi_{W,E}).
\]
The Gan--Gross--Prasad conjecture for $\mathrm{U}(n+1)\times \mathrm{U}(n)$ predicts that
\[
   \bigg| \int_{H(F)\backslash H(\mathbb{A}_F)} \varphi(h)\, \xi(h)\, \dd h \bigg|^2
   \;\approx\; L\!\left(\tfrac{1}{2}, \pi_{V,E}\times \pi_{W,E}\right).
\]

For the unitary case $\mathrm{U}(n+1)\times \mathrm{U}(n)$, the Gan--Gross--Prasad conjecture is currently known in general (see \cite{Beuzart2021a, Beuzart2021b, BC2025, BCZ2022, BLZZ2019, Yun2011, Zhang2014} for further details).

In the above cases, when the corank is larger than two, the situation becomes considerably more intricate. For example, in order to generalize Hecke’s integral for $\GL(2)\times \GL(1)$ to $\GL(n)\times \GL(1)$, one needs to apply a projection operator from cusp forms on $\GL(n)$ to cusp forms on the mirabolic subgroup of $\GL(2)$. This projection is realized by taking suitable Fourier coefficients. When $n=2$, this projection operator is precisely the identity (see \cite{JPSS1983} for details).

In analogy with the Rankin--Selberg case, one may also expect refined and generalized Gan--Gross--Prasad conjectures for $\SO(m)\times \SO(n)$ and $\mathrm{U}(m)\times \mathrm{U}(n)$. In this paper, we focus on the special case $(m,n)=(5,2)$, which connects global Bessel periods with the central value of degree-four spinor $L$-functions. For more general discussions on Bessel periods, see \cite{Liu2016}.

\subsection{Bessel periods}

Let $G=\mathrm{PGSp}(4) \simeq \mathrm{SO}(3,2)$. 
Let $\mathcal{H}_2$ denote the Siegel upper half-space of degree two, i.e.~the set of $2\times 2$ symmetric complex matrices with positive definite imaginary part. The group
\[
   G(\mathbb{R})^+ = \{ g \in G(\mathbb{R}) : \nu(g) > 0 \}
\]
acts on $\mathcal{H}_2$ via
\[
   g \langle Z \rangle = (A Z + B)(C Z + D)^{-1}, 
   \quad 
   g = \begin{pmatrix} A & B \\ C & D \end{pmatrix} 
  \in G(\mathbb{R})^+, \; Z \in \mathcal{H}_2,
\]
and the corresponding factor of automorphy is defined by
\[
  J(g,Z) = C Z + D.
\]

Let $(\rho, V_\rho)$ be an algebraic representation of $\GL_2(\mathbb{C})$. A holomorphic map $\Phi : \mathcal{H}_2 \to V_\rho$ is called a Siegel cusp form of weight $\rho$ with respect to $\Gamma_0(N)$ if it vanishes at the cusps and satisfies
\[
   \Phi(\gamma \langle Z \rangle) 
   = \rho(J(\gamma,Z)) \, \Phi(Z),
   \quad 
   \gamma \in \Gamma_0(N),\; Z \in \mathcal{H}_2.
\]

We denote by $S_\rho(\Gamma_0(N))$ the complex vector space of Siegel cusp forms of weight $\rho$ with respect to $\Gamma_0(N)$. Any $\Phi \in S_\rho(\Gamma_0(N))$ admits a Fourier expansion
\[
  \Phi(Z) = 
   \sum_{T>0} a(T,\Phi) 
   \exp\!\bigl(2\pi i \, \mathrm{tr}(TZ)\bigr),
   \quad Z \in \mathcal{H}_2,
\]
where $a(T,\Phi)\in V_\rho$ and $T$ runs over positive definite half-integral symmetric $2\times 2$ matrices.

It is well known that irreducible algebraic representations of $\GL_2(\mathbb{C})$ are parametrized by
\[
   \kappa=(n_1,n_2)\in L
  = \{(n_1,n_2)\in \mathbb{Z}^2 : n_1 \geq n_2\},
\]
and are given explicitly by
\[
   \rho_\kappa := \mathrm{Sym}^{\,n_1-n_2}\otimes {\det}^{n_2}.
\]

Let $S$ be a symmetric $2\times 2$ matrix anisotropic over $\mathbb{Q}$, so that its discriminant $d=\mathrm{disc}(S)$ is not a square in $\mathbb{Q}$. Following Dickson, Pitale, Saha, and Schmidt~\cite{DPSS2020}, we define a subgroup $T_S$ of $\GL(2)$ by
\[
   T_S(\mathbb{Q})
   = \bigl\{ g \in \GL(2,\mathbb{Q}) : {}^t g S g = \det(g)\, S \bigr\}.
\]
Then $T_S(\mathbb{Q}) \simeq K^{\times}$, where $K=\mathbb{Q}(\sqrt{d})$.

We embed $T_S$ into $\mathrm{GSp}(4)$ via
\[
   g \longmapsto
   \begin{pmatrix}
      g & 0 \\
      0 & \det(g)\, {}^t g^{-1}
   \end{pmatrix},
   \qquad g \in T_S.
\]

Next, define the subgroup $N \subseteq \mathrm{GSp}(4)$ by
\[
   N = \Bigl\{ n(X) =
   \begin{pmatrix}
      1 & X \\
      0 & 1
   \end{pmatrix}
   \;\Big|\; {}^t X = X \Bigr\}.
\]
Then the Bessel subgroup is given by $H = T_S N$, which is an extension of $T_S \simeq \SO(2)$ by the unipotent subgroup $N$.

Fix a non-trivial additive character $\psi$ of $\mathbb{Q}\backslash \mathbb{A}_\mathbb{Q}$, and define a character $\theta_S$ on $N(\mathbb{A})$ by
\[
   \theta_S(n(X)) = \psi(\mathrm{tr}(S X)).
\]

Let $\Lambda$ be a character of 
\[
   T_S(\mathbb{Q}) \backslash T_S(\mathbb{A}_\mathbb{Q}) \;\simeq\; K^\times \backslash \mathbb{A}_K^\times,
\]
such that its restriction to $\mathbb{A}_\mathbb{Q}^\times$ is trivial, i.e.~$\Lambda|_{\mathbb{A}_\mathbb{Q}^\times}=1$.

\medskip

Let $\pi$ be a unitary cuspidal generic automorphic representation of $\mathrm{GSp}_4(\mathbb{A}_\mathbb{Q})$ with trivial central character, and let $\varphi \in \pi$ be a cusp form. For $d<0$ a fundamental discriminant, and $\Lambda$ a character of $\mathrm{Cl}_K$ with $K=\mathbb{Q}(\sqrt{d})$, we define the global Bessel period
\begin{equation}\label{eqn:defofbessel}
  B_K(\varphi,\Lambda;\psi) =
  \int_{\mathbb{A}_K^{\times} 
     T_S(K)\backslash T_S(\mathbb{A}_K)}
  \int_{N(K)\backslash N(\mathbb{A}_K)}
     \varphi(tn)\,\Lambda^{-1}(t)\,
     \theta_S^{-1}(n)\, \dd n \, \dd t.
\end{equation}

Let $\mathcal{AI}(\Lambda^{-1})$ denote the automorphic representation of $\GL_2(\mathbb{A}_\mathbb{Q})$ obtained by automorphic induction of $\Lambda^{-1}$ from $\mathbb{A}_K^\times$. It is generated by the adelization of the weight-one theta series
\[
   \theta_{\Lambda^{-1}}(z) = 
   \sum_{0 \neq \mathfrak{a}\subset O_K} 
      \Lambda^{-1}(\mathfrak{a}) \,
      e^{2\pi i N(\mathfrak{a}) z}.
\]

Hence, the Gan--Gross--Prasad conjecture for $\SO(5)\times \SO(2)$ predicts that
\[
   \bigl| B_K(\varphi,\Lambda;\psi) \bigr|^2 
   \;\approx\; 
   L\!\left( \tfrac{1}{2}, \pi \times \mathcal{AI}(\Lambda^{-1}) \right).
\]

\section{Periods and \texorpdfstring{$L$}{L}-functions}\label{sec:GGP}

Any real primitive character with the modulus $d$ must be of the form $\chi_d(\cdot)=(\frac{d}{\cdot})$, where $d$ is a fundamental discriminant~\cite[Theorem 9.13]{MV2007}, i.e., a product of pairwise coprime factors of the form $-4$, $\pm 8$ and $(-1)^{\frac{p-1}{2}}p$, where $p$ is an odd prime.

Recall that $f$ is a cusp form on $\GSp_4$. Let $\pi = \otimes_v \pi_v$ be the irreducible, unitary, cuspidal automorphic representation of $\GSp_4(\mathbb{A}_\mathbb{Q})$ with trivial central character associated with $f$. For each automorphic form $\phi$ in the space of $\pi$, we define the Petersson norm
\[
   \langle \phi,\phi \rangle 
   = \int_{Z(\mathbb{A}_\mathbb{Q})\GSp_4(\mathbb{Q}) \backslash \GSp_4(\mathbb{A}_\mathbb{Q})} 
      \lvert \phi(g) \rvert^2 \, \dd g,
\]
where the measure on $Z(\mathbb{A}_\mathbb{Q}) \GSp_4(\mathbb{Q}) \backslash \GSp_4(\mathbb{A}_\mathbb{Q})$ is the Tamagawa measure.

Furusawa and Morimoto (\cite{FurusawaMorimoto2021} for the case $\Lambda=1$, and \cite{FurusawaMorimoto2024} in general) proved that
\begin{equation}\label{eqn:period}
   \frac{\lvert B_K(f,1;\psi) \rvert^2}{\langle \phi,\phi \rangle}
   = \frac{C_T}{S_\pi} \cdot 
     \frac{\zeta_{\mathbb{Q}}(2)\zeta_{\mathbb{Q}}(4)
           L(1/2,\pi)L(1/2,\pi \times \chi_d)}
          {L(1,\pi,\mathrm{Ad})\,L(1,\chi_d)} 
     \cdot \prod_v J_v(\pi_v),
\end{equation}
where the notation is explained in the proof of Theorem~\ref{thm:indep} and the local factor $J_v$ is defined in \cite{DPSS2020} and \cite{Liu2016}. We note that $J_v(\pi_v)=1$ for almost all finite places. Hence the product over places $v$ is a finite product.

\begin{remark}
A more general refined global Gan--Gross--Prasad conjecture for Bessel periods was formulated by Liu~\cite{Liu2016}. Since $\mathrm{PGSp}_4 \cong \mathrm{SO}(3,2)$, Equation~\eqref{eqn:period} represents a special case of the refined global Gan--Gross--Prasad conjecture for the pair $(\mathrm{SO}(5), \mathrm{SO}(2))$.
\end{remark}

Then our proof of Theorem~\ref{thm:indep} reduce to estimating the $1/2$-th moment of products of quadratic twisted $L$-functions.

\subsection{Mixed moments for quadratic families}\label{sec:quadratic}

Selberg~\cite{Selberg1946} demonstrated that on the critical line, the real and imaginary parts of the logarithm of the Riemann zeta function are distributed like real Gaussian variables with mean zero and variance $\frac{1}{2} \log \log T$, as the imaginary part ranges from $T$ to $2T$.
Montgomery~\cite{Montgomery1973} provided a prediction regarding the pair correlation of zeros of the Riemann zeta function on the critical line, describing the expected distribution of their spacings, then Dyson observed that this predicted distribution matches the pair correlation statistics of eigenvalues from the Gaussian Unitary Ensemble in random matrix theory.

Katz and Sarnak~\cite{KatzSarnak1999} extended this framework to families of $L$-functions over function fields, demonstrating that the distribution of their low-lying zeros corresponds to orthogonal, unitary, or symplectic symmetry types, depending on the family. Building on these ideas, Keating and Snaith~\cite{KeatingSnaith2000a,KeatingSnaith2000b} used the Circular Ensemble from random matrix theory to conjecture that the logarithm of the central values of $L$-functions in various families follows different Gaussian distributions, each with explicit leading constants.
This significantly generalized prior results by Conrey--Ghosh~\cite{ConreyGhosh1992,ConreyGhosh1998} and Balasubramanian--Conrey--Heath-Brown~\cite{BCH1985}.
Their framework was subsequently refined and extended by Conrey--Farmer~\cite{ConreyFarmer2000}, Diaconu--Goldfeld--Hoffstein~\cite{DGH2003}, Conrey--Farmer--Keating--Rubinstein--Snaith~\cite{CFKRS2005}.
A corrected and clarified version in the function field setting was later given by Sawin~\cite{Sawin2020}.

Current research on the central values of $L$-functions focuses primarily on three fronts: deriving asymptotic formulas for low-order moments, establishing lower bounds~\cite{RS2006}, and obtaining upper bounds. For the latter, Soundararajan~\cite{Soundararajan2009} developed a method based on the Generalized Riemann Hypothesis (GRH) that achieves the predicted order (up to a small power of $\log T$), while Harper's refinement~\cite{Harper2013} removes the extraneous $\varepsilon$-power entirely.

In particular, for orthogonal families of $L$-functions, the moments of central values in the range between order $0$ and $1$ (excluding the endpoints) contribute a negative power of $\log$ in their leading-order asymptotics.
This negative logarithmic contribution has concrete applications in arithmetic problems, as it leads to nontrivial savings in analytic estimates.

Traditionally, studies have focused on twisting a single automorphic form.
Building on Chandee's work~\cite{Chandee2011} on shifted moments of the Riemann zeta function, and later the work of Milinovich and Turnage-Butterbaugh~\cite{MTB2014} on integral moments of product of $L$-functions,
it has been observed that when multiple automorphic forms share the same symmetry type under a common twist, their mixed moments, i.e., products of their $L$-values, exhibit statistical independence under suitable conditions.

The combination of the decorrelation phenomenon observed in mixed moments and the negative logarithmic power bounds predicted by the Keating--Snaith conjecture for low-order moments ($0<k<1$) has emerged as a powerful framework in modern analytic number theory.
This interplay plays a central role in several recent breakthroughs.
For instance, Lester and Radziwi\l\l~\cite{LesterRadziwill2020} proved quantum unique ergodicity for half-integral weight automorphic forms, Huang and Lester~\cite{HuangLester2023} investigated the quantum variance of dihedral Maass forms, while Blomer, Brumley, and Khayutin~\cite{BlomerBrumleyKhayutin2022} proved the joint equidistribution conjecture proposed by Michel and Venkatesh in their 2006 ICM proceedings article~\cite{MichelVenkatesh2006}.
Blomer and Brumley~\cite{BlomerBrumley2024} subsequently proved the joint equidistribution of orbits in arithmetic quotients.
J\"a\"asaari, Lester, and Saha~\cite{JLS2023} established sign changes for coefficients of Siegel cusp forms of degree 2, and showed that the mass of Saito--Kurokawa lifted holomorphic cuspidal Hecke eigenforms for $\Sp_4(\mathbb{Z})$ equidistributes on the Siegel modular variety as the weight tends to infinity~\cite{JLS2024}.
Hua, Huang, and Li~\cite{HuaHuangLi2024} established a case of their joint Gaussian moment conjecture (with the holomorphic version discussed in Huang~\cite{Huang2024}). More recently, Chatzakos, Cherubini, Lester, and Risager~\cite{CCLS2025} obtained a logarithmic improvement on Selberg's longstanding bound for the error term in the hyperbolic circle problem over Heegner points with varying discriminants. Hua~\cite{Hua2025quad} demonstrated that for $0 < p < 2$, the $\ell^p$-norm of some quadratic forms in holomorphic Hecke cusp forms tends to zero asymptotically with respect to expansion in a given orthonormal basis of Hecke eigenforms.
See also Hua~\cite{Hua2025quadUni} for several applications related to quadratic twists of $\GL_2$ forms.

The study of central values of $L$-functions related to quadratic characters is an interesting topic (see e.g.~\cite{BP2022,DGH2003,DW2021,GZ2023,GH1985,HS2022,HH2023a,
HH2023b,HH2023c,Jutila1981,RS2015,Shen2022,Sono2020,Soundararajan2000,
Soundararajan2008,Soundararajan2021,SY2010,Young2009,Young2013}).
For the mixed moments of the central values of quadratic twisted symmetric cube $L$-functions associated to multiple forms, we have the following conditional result.

\begin{theorem}[Decorrelation of quadratic twisted $L$-functions]\label{thm:Lbound}
Under the same conditions as in Theorem~\ref{thm:indep}.
Then for any $\ell_1, \dots, \ell_m > 0$ and any $\varepsilon > 0$, we have
\begin{equation}\label{eqn:uppersymL}
\frac{1}{D}
\sum_{\substack{D \le -d \le 2D \\ \mu(|d|)^2 = 1\\
d \equiv a \pmod{M_0}}}
\prod_{i=1}^{m}
L\left(\tfrac{1}{2}, \mathrm{sym}^3 f_i \times \chi_d\right)^{\ell_i}
\ll_{f_1,\dots,f_m,\ell_1,\dots,\ell_m,\varepsilon}
(\log D)^{\sum_{i=1}^{m} \frac{\ell_i(\ell_i - 1)}{2} + \varepsilon}
\end{equation}
as $D \to \infty$.
\end{theorem}

\begin{remark}
  For each $1 \le i \le m$, let $\mathrm{sym}^3 f_i$ have root number $\epsilon_{\mathrm{sym}^3 f_i}$.
Let $a$ be an integer such that, for every fundamental discriminant $d<0$ with $d \equiv a \pmod{M_0}$, the root numbers satisfy
\begin{equation*}
\epsilon_{\mathrm{sym}^3 f_i}(d) = \epsilon_{\mathrm{sym}^3 f_i} \chi_d(-M_i) = 1, \quad \text{for all } i = 1, \dots, m.
\end{equation*}
If this condition fails, then at least one of the $m$ associated $L$-functions vanishes identically.
\end{remark}

\begin{remark}
  The $\varepsilon$-power can be removed using Harper's method~\cite{Harper2013}, see also~\cite[Theorem 2]{Hagen2024}.
\end{remark}

We list some preliminaries in \S\ref{sec:lemma}, and prove Theorem~\ref{thm:Lbound} in \S\ref{sec:proof}.

\subsection{Proof of Theorem~\ref{thm:indep}}\label{sec:proofofmain}

\begin{proof}
The proof relies on the explicit relation between the Bessel period and automorphic $L$-functions (see Equation~\eqref{eqn:period}). Before proceeding, we clarify some notations. 

In Equation~\eqref{eqn:period}, $\zeta_{\mathbb{Q}}(s) := \pi^{-s/2}\Gamma(s/2)\zeta(s)$ denotes the completed Riemann zeta function, and $S_\pi$ is a certain integral power of $2$ related to the Arthur parameter of $f$. In particular, $S_\pi = 4$ if $L(s,\pi) = L(s,\pi_1)L(s,\pi_2)$ for irreducible cuspidal representations $\pi_i$ of $\GL_2(\mathbb{A}_{\mathbb{Q}})$, and $S_\pi = 2$ if $L(s,\pi) = L(s,\Pi)$ for an irreducible cuspidal representation $\Pi$ of $\GL_4(\mathbb{A}_{\mathbb{Q}})$. 

The measure $dt$ denotes the global Tamagawa measure, and $C_T$ is defined via the factorization $\dd t = C_T \prod_v \dd t_v$. Moreover, $J_v(\pi_v)$ is the local factor as defined in \cite{DPSS2020,FurusawaMorimoto2024}. For almost all finite places $v$, the local factor equals one.

By the assumptions in Theorem~\ref{thm:indep}, we know that
\[
   L(1/2, \sym^3 f_i) \ge 0 \quad \text{and} \quad 
   L(1/2, \sym^3 f_i \times \chi_d) \ge 0.
\]
Under the GRH, we also have 
\[
   L(1, \chi_d) \gg (\log \log |d|)^{-1}.
\]

Hence, combining Equation~\eqref{eqn:period} with Theorem~\ref{thm:Lbound}, we obtain the estimate stated in Theorem~\ref{thm:indep}.
\end{proof}

\section{Preliminaries}\label{sec:lemma}

\begin{lemma}\label{lemma:sumusesummationformula}
Let $r \in \mathbb{N}$.
Then, for $x \le D^{1/(10r)}$ and any real numbers $a_p \ll p^{\varepsilon}$ for any $\varepsilon > 0$, we have
\begin{equation*}
    \sum_{\substack{D \leq -d \leq 2D\\ \mu(|d|)^2 = 1\\ d \equiv a \pmod{N_0}}}
    \left( \sum_{p \le x} \frac{a_p \chi_d(p)}{p^{1/2}} \right)^{2r}
    \ll \frac{(2r)!}{r! \, 2^r} D \left( \sum_{p \le x} \frac{a_p^2}{p} \right)^r.
\end{equation*}
\end{lemma}

\begin{proof}
Similarly to the proof of \cite[Lemma 4.3]{LesterRadziwill2020}.
\end{proof}

Let $\alpha_{f_i}, \beta_{f_i}$ denote the Satake parameters of $f_i$.
For $p \nmid N_i$, define
\begin{equation}\label{def:Lambdafgh}
\Lambda_{\mathrm{sym}^3 f_i \times \chi_d}(p^n)
= \bigl( \alpha_{f_i}(p)^{3n} + \alpha_{f_i}(p)^n + \beta_{f_i}(p)^n + \beta_{f_i}(p)^{3n} \bigr) \chi_d(p^n).
\end{equation}
In particular, we have the Hecke relation
\begin{equation}\label{eqn:heckerelation}
\Lambda_{\mathrm{sym}^3 f_i \times \chi_d}(p^2)
= \left( \lambda_{\mathrm{sym}^6 f_i}(p)
-\lambda_{\mathrm{sym}^4 f_i}(p)
+\lambda_{\mathrm{sym}^2 f_i}(p) - 1 \right) \delta_{p \nmid d}.
\end{equation}

\begin{lemma}[{\cite[Theorem 2.1]{Chandee2009}}]\label{lemma:LogLfunctions}
Assume GRH for $L(s, \mathrm{sym}^3 f_i \times \chi_d)$. Then for $x > 10$, we have
\begin{equation*}
\log L\left( \tfrac{1}{2}, \mathrm{sym}^3 f_i \times \chi_d \right)
\le \sum_{p^n \le x} \frac{\Lambda_{\mathrm{sym}^3 f_i}(p^n) \chi_d(p^n)}
{n p^{n \left( \frac{1}{2} + \frac{1}{\log x} \right)}}
\cdot \frac{\log \frac{x}{p^n}}{\log x}
+ O_{f_i} \left( \frac{\log |d|}{\log x} + 1 \right),
\end{equation*}
where the implied constant is absolute.
\end{lemma}

\begin{lemma}\label{lemma:sumofcoefficients}
Let $1 \le i \le m$, and suppose each $f_i$ is an even-weight Hecke eigenform.
Assume that the forms $f_1, \dots, f_m$ are pairwise distinct.
Assuming the RH, and the GRH for $L(s, \sym^3 f_i \times \sym^3 f_j)$ with $i\neq j$, and for $L(s, \sym^k f_i)$ with $k=2,4,6$ and all $i$.

Then for any $\ell_1, \dots, \ell_m > 0$ and for $2 \le y \le x$, we have
\begin{equation*}
\sum_{y < p \le x} \frac{\left( \sum_{i=1}^{m} \ell_i \lambda_{\mathrm{sym}^3 f_i}(p) \right)^2 \delta_{p \nmid d}}{p}
= \left( \sum_{i=1}^{m} \ell_i^2 \right) \log \frac{\log x}{\log y}
+ O_{f_1, \dots, f_m}(\log\log\log |d|).
\end{equation*}
\end{lemma}

\begin{proof}
For algebraic regular Hecke eigenform $f_i$,
from Gelbart--Jacquet~\cite{GelbartJacquet1978},
each $\sym^2 f_i$ is a self-dual cuspidal form,
from Kim--Shahidi~\cite{KimShahidi2001,KimShahidi2002a,KimShahidi2002b}, each $\sym^3 f_i$ is a self-dual cuspidal form,
and from Newton--Thorne~\cite{NewtonThorne2021a,NewtonThorne2021b}, $\sym^4 f_i$ and $\sym^6 f_i$ are self-dual cuspidal forms.
Then from Selberg's orthogonality~\cite{LiuWangYe2005} of $\sym^6 f_i \times \sym^2 f_j$ and $\sym^4 f_i \times \sym^2 f_j$ for $i\neq j$, there is no pole at $s=1$ for $L(s,\sym^5 f_i\times (f_i\boxtimes \sym^2 f_j))$, and by Ramakrishnan~\cite{Ramakrishnan2011,Ramakrishnan2015}, the $\sym^3 f_i$ are pairwise distinct.

Assuming the GRH for $L(s, \sym^3 f_i \times \sym^3 f_j)$ with $i\neq j$, the function $\log L(s, \sym^3 f_i \times \sym^3 f_j)$ is analytic for $\Re(s) \geq \frac{1}{2} + \frac{1}{\log x}$.
 By a classical argument of Littlewood \cite[(14.2.2)]{Titchmarsh1986}, in this region we have
  \begin{equation}\label{eqn:logLbound}
    |\log L(s, \sym^3 f_i \times \sym^3 f_j)|
    \ll \left(\Re(s) - \frac{1}{2}\right)^{-1} \log |\Im(s)|.
  \end{equation}

  For $\Re(s) > 0$, we have
  \begin{equation*}
    \sum_{n} \frac{|\Lambda_{\sym^3 f_i}(n) \Lambda_{\sym^3 f_j}(n)|}{n^{1+s}} \ll 1,
  \end{equation*}
  and Deligne's bound yields
  \begin{equation*}
    \sum_{a \geq 2} \sum_{p^a \leq x} \frac{|\Lambda_{\sym^3 f_i}(p^a)
    \Lambda_{\sym^3 f_j}(p^a)|}{p^a} \ll 1.
  \end{equation*}

  Applying Perron's formula for $x \geq 2$ gives
  \begin{multline*}
    \sum_{p \leq x} \frac{\lambda_{\sym^3 f_i}(p)\lambda_{\sym^3 f_j}(p)}{p}
    = \sum_{p \leq x} \frac{\Lambda_{\sym^3 f_i}(p)\Lambda_{\sym^3 f_j}(p)}{p} \\
    = \frac{1}{2\pi i} \int_{1 - ix\log x}^{1 + ix\log x}
    \log L(s + 1, \sym^3 f_i \times \sym^3 f_j) x^s \frac{\dd s}{s} \\
    + O\left(\frac{x \sum_{p \text{ prime}} \frac{|\lambda_{\sym^3 f_i}(p)\lambda_{\sym^3 f_j}(p)|}{p^2}}{x \log x}\right)
    + O(1).
  \end{multline*}

  Shifting the contour to $\Re(s) = -\frac{1}{2} + \frac{1}{\log x}$, we encounter a simple pole at $s = 0$ with residue $\log L(1, \sym^3 f_i \times \sym^3 f_j)$. The upper horizontal contour is bounded by
  \begin{multline*}
    \ll \frac{1}{x \log x}
    \int_{-\frac{1}{2} + \frac{1}{\log x} + ix \log x}^{1 + ix \log x}
    |\log L(s + 1, \sym^3 f_i \times \sym^3 f_j)| |x^s| |\dd s| \\
    \ll \frac{\log x \log(x \log x)}{x \log x}
    \int_{-\frac{1}{2}}^{1} x^u \dd u \ll 1,
  \end{multline*}
  and similarly for the lower horizontal contour.

  From \eqref{eqn:logLbound}, we obtain for $x \geq 2$:
  \begin{multline*}
    \sum_{p \leq x} \frac{\lambda_{\sym^3 f_i}(p)\lambda_{\sym^3 f_j}(p)}{p}
    = \log L(1, \sym^3 f_i \times \sym^3 f_j) \\
    + O\left(1 + \frac{\log x}{\sqrt{x}} \int_{-x \log x}^{x \log x}
    \frac{\log u}{1 + |u|} \dd u \right).
  \end{multline*}

  Applying this estimate twice yields for $z \geq 3$:
  \begin{equation*}
    \left| \sum_{2 < p \leq z} \frac{\lambda_{\sym^3 f_i}(p)\lambda_{\sym^3 f_j}(p)}{p} \right| \ll 1.
  \end{equation*}

  Similarly, for $k=2,4,6$, we have
    \begin{equation*}
    \left| \sum_{2 < p \leq z} \frac{\lambda_{\sym^k f_i}(p)}{p} \right| \ll 1.
  \end{equation*}
  Then we finish the proof with using $\sym^3 f_i(p)^2=1+\sum_{k=1}^{3}\sym^{2k} f_i(p)$.
\end{proof}

\begin{remark}
Here we use Selberg's orthogonality, which relies on Deligne's bound.
For more general automorphic forms, one typically requires Rudnick and Sarnak's Hypothesis H, which was recently proved by Jiang~\cite{Jiang2025}.
\end{remark}

Before stating the next lemma, we introduce the following notation.
For parameters $2 \le y \le x$, define
\begin{equation*}
\mathcal{P}(d; x, y) = \sum_{p \le y} \frac{\left( \sum_{i=1}^{m} \ell_i \lambda_{\mathrm{sym}^3 f_i}(p) \right) \chi_d(p)}{p^{\frac{1}{2} + \frac{1}{\log x}}} \left(1 - \frac{\log p}{\log x} \right),
\end{equation*}
and define
\begin{align*}
\mathcal{A}(V; x) &= \#\left\{ D \le -d \le 2D : \mu(|d|)^2 = 1,\ d \equiv a \pmod{N_0},\ \mathcal{P}(d; x, x) > V \right\}, \\
\sigma^2(D) &= \left( \sum_{i=1}^{m} \ell_i^2 \right) \log\log D.
\end{align*}

\begin{lemma}\label{lemma:AXY}
Assume the setup above.
Let $C \ge 1$ be fixed and let $\varepsilon > 0$ be sufficiently small.
Then, for all
\[
\sqrt{\log\log D} \le V \le C \frac{\log D}{\log\log D},
\]
we have the bound
\begin{equation*}
\mathcal{A}\left(V; D^{\frac{1}{\varepsilon V}} \right)
\ll_{f_1, \dots, f_m} D \left(
e^{-\frac{(1 - 2\varepsilon)V^2}{2\sigma(D)^2}} (\log\log D)^3
+ e^{-\frac{\varepsilon}{11} V \log V}
\right).
\end{equation*}
\end{lemma}

\begin{proof}
Throughout the proof, we assume that $\varepsilon > 0$ is sufficiently small, such that $\varepsilon V$ remains small.
We restrict to the range
\[
\sqrt{\log\log D} \le V \le C \frac{\log D}{\log\log D}.
\]

Following the optimization method of Soundararajan, we set the length of the Dirichlet polynomial to be $x = D^{1/(\varepsilon V)}$.
We decompose $\mathcal{P}(d; x, x)$ as $\mathcal{P}_1(d) + \mathcal{P}_2(d)$, where $\mathcal{P}_1(d) = \mathcal{P}(d; x, z)$ with $z = x^{1/\log\log D}$.
This choice ensures $\sum_{z \le p \le x} \frac{1}{p} \ll \log\log\log D$.

Let $V_1 = (1 - \varepsilon)V$ and $V_2 = \varepsilon V$.
Then, if $\mathcal{P}(d; x, x) > V$, we must have either
\begin{equation}\label{eqn:case1}
\mathcal{P}_1(d) > V_1,
\end{equation}
or
\begin{equation}\label{eqn:case2}
\mathcal{P}_2(d) > V_2.
\end{equation}

Using Lemma~\ref{lemma:sumusesummationformula} and Lemma~\ref{lemma:sumofcoefficients}, we find that for parameters satisfying $r \le \frac{\varepsilon V}{10} \log\log D$ and $z \ll D^{1/(10r)}$, the number of squarefree integers $d$ with $D \le -d \le 2D$ and $d \equiv a \pmod{N_0}$ satisfying \eqref{eqn:case1} is bounded by
\begin{equation*}
\frac{1}{V_1^{2r}} \sum_{\substack{D \le -d \le 2D\\ \mu(|d|)^2 = 1\\ d \equiv a \pmod{N_0}}} \mathcal{P}_1(d)^{2r}
\ll \frac{(2r)!}{V_1^{2r} r! 2^r} D (\log\log D)^3 \sigma(D)^{2r}.
\end{equation*}

We now consider two cases for the parameter $r$:
\begin{itemize}
\item If $V \le \frac{\varepsilon}{10} \sigma(D)^2 \log\log D$, set $r = \left\lfloor \frac{V_1^2}{2\sigma(D)^2} \right\rfloor$.
\item Otherwise, set $r = \left\lfloor \frac{\varepsilon V}{10} \right\rfloor$.
\end{itemize}

This yields the estimate
\begin{multline*}
\#\left\{ D \le -d \le 2D : \mu(|d|)^2 = 1,\ d \equiv a \pmod{N_0},\ \mathcal{P}_1(d) > V_1 \right\}
\\ \ll_{f_1, \dots, f_m} D \left(
e^{-\frac{(1 - 2\varepsilon)V^2}{2\sigma(D)^2}} (\log\log D)^3
+ e^{-\frac{\varepsilon}{11} V \log V}
\right).
\end{multline*}

To bound the number of $d$ satisfying \eqref{eqn:case2}, we take $r = \left\lfloor \frac{\varepsilon V}{10} \right\rfloor$.
Note that $x \ll D^{1/(10r)}$.
Using Lemma~\ref{lemma:sumusesummationformula} and Lemma~\ref{lemma:sumofcoefficients} again, we obtain
\begin{equation*}
\frac{1}{V_2^{2r}} \sum_{\substack{D \le -d \le 2D\\ \mu(|d|)^2 = 1\\ d \equiv a \pmod{N_0}}} \mathcal{P}_2(d)^{2r}
\ll D (\log\log D)^3 \frac{(2r)!}{r!}
\left( \frac{C}{V_2^2} \log\log\log D \right)^r
\ll_{f_1, \dots, f_m} D e^{-\frac{\varepsilon}{11} V \log V}.
\end{equation*}

Combining the two estimates completes the proof.
\end{proof}

\section{Proof of Theorem~\ref{thm:Lbound}}\label{sec:proof}

We establish Theorem~\ref{thm:Lbound} by applying Soundararajan's method \cite{Soundararajan2009}.

\begin{proof}
Using the relation~\eqref{eqn:heckerelation} and bounding the contribution from terms with $n \geq 3$, we obtain the decomposition
\begin{multline}\label{eqn:Lambdasumsym2fuj}
    \sum_{p^n \leq x} \frac{\Lambda_{\mathrm{sym}^3 f_i \times \chi_d}(p^n)}{n p^{n(\frac{1}{2}+\frac{1}{\log x})}} \frac{\log \frac{x}{p^n}}{\log x}
    = \sum_{p \leq x} \frac{\lambda_{\mathrm{sym}^3 f_i}(p) \chi_d(p)}{p^{\frac{1}{2}+\frac{1}{\log x}}} \frac{\log \frac{x}{p}}{\log x} \\
    + \frac{1}{2} \sum_{p \leq \sqrt{x}} \frac{(\lambda_{\sym^6 f_i}(p)-\lambda_{\sym^4 f_i}(p)+\lambda_{\sym^2 f_i}(p)-1)\delta_{p\nmid d}}{p^{1+\frac{2}{\log x}}} \frac{\log \frac{x}{p^2}}{\log x}
    + O_{f_i}(1).
\end{multline}

Applying Lemma~\ref{lemma:sumofcoefficients} to the second sum in~\eqref{eqn:Lambdasumsym2fuj} yields
\begin{equation}\label{eqn:secondterms2}
    -\frac{1}{2}\log\log x + O_{f_i}(\log\log\log D).
\end{equation}

Define the key quantities
\begin{equation*}
    \eta(D) := \left(-\frac{1}{2}+\varepsilon\right) \left(\sum_{i=1}^{m}\ell_i\right) \log\log D,
\end{equation*}
and
\begin{equation*}
    \mathcal{L}(d) := \prod_{i=1}^{m} L(\frac{1}{2}, \mathrm{sym}^3 f_i \times \chi_d)^{\ell_i},
\end{equation*}
with the counting function
\begin{equation*}
    \mathcal{B}(V) := \#\left\{D\leq -d\leq 2D : \mu(|d|)^2=1,~d\equiv a\!\!\!\!\pmod{N_0},~\log \mathcal{L}(d) > V \right\}.
\end{equation*}

By integration by parts, we have
\begin{equation*}
    \sum_{\substack{D\leq\sigma d\leq 2D\\ \mu(|d|)^2=1\\
    d\equiv a\!\!\!\!\pmod{N_0}}} \mathcal{L}(d)
    = -\int_{\mathbb{R}} e^V \, d\mathcal{B}(V)
    = \int_{\mathbb{R}} e^V \mathcal{B}(V) \, dV
    = e^{\eta(D)} \int_{\mathbb{R}} e^V \mathcal{B}(V+\eta(D)) \, dV.
\end{equation*}

Under GRH, a Littlewood-type bound (see~\cite[Corollary 1.1]{Chandee2009} or~\cite[\S 4]{ChandeeSoundararajan2011}) implies
\begin{equation*}
    \log \mathcal{L}(d) \leq C \frac{\log D}{\log\log D}
\end{equation*}
for some constant $C > 1$. Therefore, in the above integral, we may restrict to the range
\begin{equation*}
    \sqrt{\log\log D} \leq V \leq C\frac{\log D}{\log\log D},
\end{equation*}
while for smaller $V$ we trivially bound the contribution by $D+1$.

Set
\begin{equation*}
    x := D^{1/(\varepsilon V)}.
\end{equation*}
Then, for
\begin{equation*}
    \sqrt{\log\log D} \leq V \leq (\log\log D)^4,
\end{equation*}
we have
\begin{equation*}
    -\frac{1}{2} \left(\sum_{i=1}^{m}\ell_i\right) \log\log D + O_{f_1,\dots,f_m}(\log\log\log D) \leq \eta(D).
\end{equation*}

From Lemma~\ref{lemma:LogLfunctions} and~\eqref{eqn:secondterms2}, we deduce
\begin{equation*}
    \mathcal{B}(V+\eta(D)) \leq \mathcal{A}(V(1-2\varepsilon); x),
\end{equation*}
which holds for $\sqrt{\log\log D} \leq V \leq (\log\log D)^4$, and remains valid for $V \geq (\log\log D)^4$ since in that range $V+\eta(D) = V(1 + o(1))$.

Combining these estimates with Lemma~\ref{lemma:AXY}, we obtain for some absolute constant $C > 0$:
\begin{multline*}
    \sum_{\substack{D\leq\sigma d\leq 2D\\ \mu(|d|)^2=1\\
    d\equiv a\!\!\!\!\pmod{N_0}}} \mathcal{L}(d)
    \ll D e^{\eta(D)} \int_{\sqrt{\log\log D}}^{C\frac{\log D}{\log\log D}}
    e^V \left( e^{-\frac{(1-\varepsilon)V^2}{2\sigma(D)^2}} (\log\log D)^3 + e^{-\varepsilon V\log V} \right) dV \\
    \ll D (\log D)^\varepsilon \cdot e^{\eta(D) + \frac{\sigma(D)^2}{2}}
    \ll D (\log D)^{\frac{\sum_{i=1}^{m}\ell_i(\ell_i-1)}{2} + \varepsilon},
\end{multline*}
where in the last step we use the Gaussian integral identity
\begin{equation*}
    \int_{\mathbb{R}} e^{-\frac{x^2}{2\sigma^2} + x} \, dx = \sqrt{2\pi} \sigma \, e^{\frac{\sigma^2}{2}}.
\end{equation*}

This completes the proof.
\end{proof}

\section*{Acknowledgements}

The authors would like to sincerely thank Professor Dihua Jiang for providing helpful comments and useful references.
The authors sincerely thank Zhaolin Li for helpful discussions.
S.H. would like to thank Professors Bingrong Huang and Philippe Michel for their constant encouragement.
X.M. would like to thank Professor Valentin Blomer for his constant encouragement.
S.H. was partially supported by the National Key R\&D Program of China (No. 2021YFA1000700) and NSFC (No. 12031008).
X.M. was partially supported by the ERC Advanced Grant 101054336 and Germany's Excellence Strategy grant EXC-2047/1-390685813.
This work was completed and partially supported during the Aarhus Automorphic Forms Conference in August 2025.

\section*{Conflict of interest statement}

All authors contributed equally throughout the research process.
The authors' names are listed in alphabetical order, and all authors serve as co-corresponding authors for this submission.

%%%%%%%%%%%%%%%%%%%%%%%%%%%%%%%%%%%%%%%%%%%%%%%%%%%%%%%%%%%%%%
%%%%%                      References                    %%%%%
%%%%%%%%%%%%%%%%%%%%%%%%%%%%%%%%%%%%%%%%%%%%%%%%%%%%%%%%%%%%%%
%\medskip

\end{document}